\documentclass{article}%
\usepackage{amsfonts, amsmath, amssymb}
\usepackage{subfig}
\usepackage[dvips]{graphicx}
\usepackage{pst-plot, pstricks, pst-node}
\usepackage{cite}

\newtheorem{theorem}{Theorem}

\newtheorem{definition}[theorem]{Definition}

\newtheorem{proposition}[theorem]{Proposition}
\newtheorem{remark}[theorem]{Remark}

\newenvironment{proof}[1][Proof]{\noindent\textbf{#1.} }{\ \rule{0.5em}{0.5em}}
\begin{document}

\title{The work of Jorge Ize regarding the $n$-body problem}
\author{C. Garc\'{\i}a-Azpeitia}
\maketitle

\begin{abstract}
In this paper we present a summary of the last works of Jorge Ize regarding
the global bifurcation of periodic solutions from the equilibria of a
satellite attracted by $n$ primary bodies. We present results on the global
bifurcation of periodic solutions for the primary bodies from the Maxwell's
ring, in the plane and in space, where $n$ identical masses on a regular
polygon and one central mass are turning in a plane at a constant speed. The
symmetries of the problem are used in order to find the irreducible
representations, and with the help of the orthogonal degree theory, all the
symmetries of the bifurcating branches. The results presented in this paper
were done during the Ph.D. of the author under the direction of Jorge Ize (see
\cite{GaIz10,GaIz11,GaIz12,GaIz13,GaIz14}). This paper is dedicated to his memory.

\end{abstract}

{This paper is devoted to present the results of the author in collaboration
to Jorge Ize regarding the movement of a satellite attracted by $n$ primary
bodies}. In particular, when the primary bodies form the polygonal relative
equilibrium corresponding to $n$ identical masses arranged on a regular
polygon with one mass in the centre. This model was posed by Maxwell in order
to explain the stability of Saturn's rings.

For this polygonal relative equilibrium, we give also a description of
bifurcation of planar and spatial periodic solutions. According to the value
of the central mass, there are up to $2n$ branches of planar periodic
solutions, with different symmetries, and up to $n$ additional branches, with
non trivial vertical components, if some non resonance condition is
satisf{ied}. The linearization of the system is degenerated due to rotational
symmetries. These facts imply that the classical bifurcation results for
periodic solutions may not be applied directly. The proof is carried on with
the use of a topological degree for maps that commute with symmetries and are
orthogonal to the infinitesimal generators for these symmetries.

{We also expose} the global bifurcation of periodic solutions for a satellite
attracted by $n$ primary bodies. These solutions will form a continuum in the
plane of the primaries and other solutions outside the plane. A particular
attention is given to the case where $n+1$ primaries form the Maxwell's Saturn ring.

In order to explain the results, {we give} a short description of the steps to
prove the bifurcation theorem. The ideas { we} follow are from the book
\cite{IzVi03}, where general bifurcation theorems are proven. In addition, in
\cite{Ga10} there is a systematic application to Hamiltonian systems. The
results exposed here for the $n$-body problem and the satellite are from the
papers \cite{GaIz10}, \cite{GaIz11}, and \cite{GaIz13}.

\section{Orthogonal degree}

A Hilbert space $V$ is a $\Gamma$-\textbf{representation} if there is a
morphism of groups
\[
\rho:\Gamma\rightarrow GL(V).
\]
The action of the group over a point generate{s} one orbit denoted by $\Gamma
x$. A set $\Omega\subset V$ is $\Gamma$-\textbf{invariant} if it is made of
orbits, this is $\Gamma x\subset\Omega$ for all $x\in\Omega$.

The \textbf{isotropy group }of a point $x$ is defined by
and the \textbf{fixed point space} of the subgroup $H$ is
\[
X^{H}=\{x\in X:hx=x,~\forall h\in H\}.
\]

A space $V$ is an \textbf{irreducible representation} when $V$ does not have
$\Gamma$-invariant proper subspaces. The irreducible representations of the
action of a compact abelian Lie group are always two dimensional, and as such,
equivalent to the complex space .

A function $f:\Omega\rightarrow W$ is $\Gamma$\textbf{-equivariant }if
\[
f(\gamma x)=\gamma^{\prime}f(x)\text{,}%
\]
and $\Gamma$\textbf{-invariant} if the action in the range is trivial,
$f(\gamma x)=f(x)$.

\begin{proposition}
A differentiable $\Gamma$-equivariant function at $x$ satisf{ies}%
\[
df(\gamma x)\gamma=\gamma^{\prime}df(x)
\]
for all $\gamma\in\Gamma$. In particular, the derivative $f^{\prime}(x)$ is a
$\Gamma_{x}$--equivariant map. Moreover, the gradient of a $\Gamma$-invariant
functional is a $\Gamma$-equivariant map when the action is orthogonal.
\end{proposition}

\begin{proof}
The first statement follows from the uniqueness of the derivative, and from
the equality%
\begin{align*}
df(\gamma x)\gamma y+o(y)  &  =f(\gamma(y+x))-f(\gamma x)\\
&  =\gamma^{\prime}[f(y+x)-f(x)]=\gamma^{\prime}df(x)y+o(y).
\end{align*}
The second statement is a consequence of%
\[
\gamma^{T}\nabla f(\gamma x)=[Df(\gamma x)\gamma]^{T}=Df(x)^{T}=\nabla
f(x)\text{.}%
\]

\end{proof}

Let $\gamma$ be an element of a torus $\gamma=(\varphi_{1},...,\varphi_{n})\in
T^{n}$, with $\varphi_{j}\in(-\pi,\pi)$. The $j$-th \textbf{generator} of
\ the torus $T^{n}$ is the vector fields tangent to the orbit%
\[
A_{j}x=\frac{\partial}{\partial\varphi_{j}}(\gamma x)|_{\gamma=0}\text{.}%
\]

The gradient of a $\Gamma$-invariant function is $\Gamma$-equivariant by the
previous proposition. Moreover, this kind of gradient is orthogonal to the
generators because%
\[
\left\langle \nabla f(x),A_{j}x\right\rangle =\frac{\partial}{\partial
\varphi_{j}}f(\gamma x)|_{\gamma=0}=0.
\]

A general $\Gamma$-equivariant map is called $\Gamma$\textbf{-orthogonal} if
it satisfy
\[
\left\langle f(x),A_{j}x\right\rangle =0\text{ for all }x\in\Omega.
\]

The following definition of $\Gamma$-orthogonal degree for compact abelian Lie
groups is due to J. Ize and A. Vignoli, see \cite{IzVi99} .

Let $\Gamma$ be a compact abelian Lie group, an $\Omega$ a $\Gamma$-invariant
domain of $V$. Let $f_{0}$ an $f_{1}$ two $\Gamma$-orthogonal maps which are
non-zero on the boundary $\partial\Omega$. It is said that two maps $f_{0}$
and $f_{1}$ are $\Gamma$\textbf{-orthogonal homotopic} when there is a
continuous deformation
\[
f_{t}:\bar{\Omega}\times\lbrack0,1]\rightarrow E\text{,}%
\]
where the map $f_{t}$ is $\Gamma$-orthogonal and non-zero in the boundary
$\partial\Omega$ for each step $t$ .

The ball $B=\{x\in V:\left\Vert x\right\Vert \leq r\}$ is $\Gamma$-invariant
when the representation in $V$ is an isometry. In this case, let { us} define
$\mathcal{C}$ as the set of $\Gamma$-orthogonal maps of the form%
\[
f:\partial([0,1]\times B)\rightarrow\mathbb{R}\times V-\{0\}.
\]
Since the boundary of $[0,1]\times B_{r}$ is isomorph{ic} to the sphere
$S^{V}$, and since the set $\mathbb{R}\times V-\{0\}$ is $\Gamma$-homotopic to
$S^{V}$, then the map $f$ may be thought from $S^{V}$ into $S^{V}$.

Since the $\Gamma$-orthogonal homotopy form{s} an equivalent relation in
$\mathcal{C}$, then one define $\Pi_{\perp}[S^{V}]$ as the set of equivalent
classes of $\mathcal{C}$ and
\[
\lbrack f]_{\perp}\in\Pi_{\perp}[S^{V}]
\]
as the equivalent class of $f$ .

Shrinking the top $\{0\}\times B$ and the bottom $\{1\}\times B$ to the point
$(1,0)$, one may prove that all homotopy classes $[f]_{\perp}$ have one
function such that $f(t,x)=(1,0)$ for $t\in\{0,1\}$. With these functions one
may define the sum of homotopy classes as $[f]_{\perp}+[g]_{\perp}=[f\oplus
g]_{\perp}$ with%
\[
f\oplus g=\left\{
\begin{array}
[c]{c}%
f(2t,x)\text{ for }t\in\lbrack0,1/2],\\
g(2t-1,x)\text{ for }t\in\lbrack1/2,1].
\end{array}
\right.
\]

With this sum, the set $\Pi_{\perp}[S^{V}]$ has a group structure. The
identity is the map $[(1,0)]_{\perp}$, and the inverse of some class
$[f]_{\perp}$ is the class $[f(1-t,x)]_{\perp}$. Moreover, one may prove that
the group $\Pi_{\perp}[S^{V}]$ is abelian when $V^{\Gamma}$ is non trivial.

To define the $\Gamma$-orthogonal degree of a map $f:\Omega\rightarrow V$, it
is necessary to extend the function $f$ to a ball, $\bar{f}:$ $\Omega\subset
B\rightarrow V$. Also, one needs a Urysohn $\Gamma$-invariant map with value
$0$ in $\bar{\Omega}$, and value $1$ in $B\backslash N$, where $N $ is a small
neighborhood of $\bar{\Omega}$. The existence of the Urysohn map $\varphi$ and
the extension $\bar{f}$ follows from the $\Gamma$-orthogonal extension theorem
of Borsuk. The proof of this theorem is only for actions of compact abelian
Lie groups on finite spaces, see \cite{IzVi03} . \begin{figure}[th]
\centering
\begin{pspicture}(-2,-2)(6,2)\SpecialCoor
\psellipse(0,1.5)(2,.5)
\psellipse[linewidth=0pt,fillstyle=solid,
fillcolor=lightgray](0,0)(2,.5)
\psline(2,-1.5)(2,1.5)
\psline{->}(-2,-1.5)(-2,2)
\rput[b](-2,2){$t$}
\rput[r](-2,0){\small$1/2\;$}
\psccurve[linewidth=2pt,fillstyle=solid,
fillcolor=white](-.7,-.2)(.35,-.2)(.6,.1)(-.45,.2)
\psellipse[linewidth=2pt,fillstyle=solid,
fillcolor=lightgray](0,0)(.2,.15)
\pscircle(5,0){1}
\psellipse[linestyle=dashed](5,0)(1,.3)
\psdot(5,0)\rput[l](5,0){$\ 0$}
\psarc{<-}(3,0){1}{45}{135}
\rput(3.5,1.5){\small$(2t+2\varphi-1,\bar{f})$}
\pcline[linestyle=dashed]{->}(-.3,-1)(-.3,0)\lput*{0}(0){\small$\Omega$}
\rput(-.3,-1.2){\small$f=\bar{f}$}
\rput(-.3,-1.5){\small$\varphi=0$}
\pcline[linestyle=dashed]{->}(1.2,-1.5)(1.2,0)\lput*{0}(0){\small$\varphi=1$}
\psellipticarc(0,-1.5)(2,.5){180}{0}
\NormalCoor\end{pspicture}
\caption{Degree definition.}%
\end{figure}
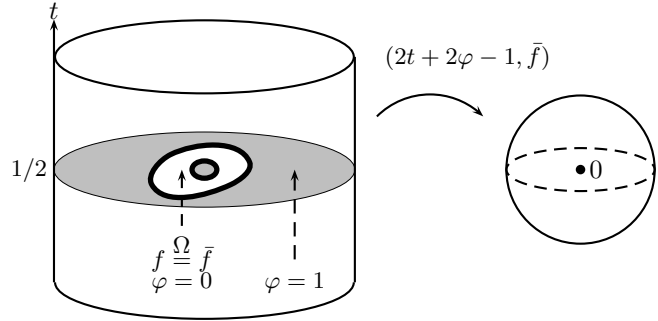

\begin{definition}
\label{1EnDeDe}The $\Gamma$-orthogonal degree of $f$\ is defined as the
homotopy class%
\[
\deg_{\perp}(f;\Omega)=[(2t+2\varphi-1,\bar{f})]_{\perp}\in\Pi_{\perp}%
[S^{V}].
\]

\end{definition}

When the domain is a ball, the degree is just the homotopy class of the
suspension $\deg_{\perp}(f;B)=[(2t-1,f)]_{\perp}$, and this definition is
equivalent to the Brouwer degree if the action of the group is trivial.

In \cite{IzVi03} it is proven that for each isotropy group of $\Gamma$, $H\in
Iso(\Gamma)$, the group $\Pi_{\perp}[S^{V}]$ has a copy of a group
isomorph{ic} to the group $\mathbb{Z}$, this is%
\[
\Pi_{\perp}[S^{V}]=\bigoplus_{H\in Iso(\Gamma)}\mathbb{Z}\text{.}%
\]

Moreover, the degree of the map $f$ is
\[
\deg_{\perp}(f;\Omega)=\sum_{H\in Iso(\Gamma)}d_{H}[F_{H}]_{\perp}\text{,}%
\]
where $[F_{H}]_{\perp}$ is the generator of one $\mathbb{Z}$ corresponding to
each isotropy group $H\in Iso(\Gamma)$, and $d_{H}$ is just an integer.

The orthogonal degree has the known properties of a degree: existence,
excision and $\Gamma$-orthogonal homotopy invariance. In this case, the
existence property means that the map $f$ must have a zero in $\Omega\cap
V^{H}$ if $d_{H}\neq0$.

\begin{remark}
For a $k$-dimensional orbit, with a tangent space generated by $k$ of the
infinitesimal generators of the group, one uses a Poincar\'{e} section for the
map augmented with $k$ Lagrange-like multipliers for the generators. (See the
construction in \cite{IzVi03}, section 4.3). For instance, for the action of
$SO(2)$, the study of zeros of the equivariant map $F(x)$, orthogonal to the
generator $Ax$, is equivalent to the study of the zeros of $F(x)+\lambda Ax$;
if $x$ is not fixed by the group, i.e., if $Ax$ is not $0$, for which
$\lambda$ is $0$. In this way, one has added an artificial parameter. This
trick has been used very often and, in the context of a topological degree
argument, was called \textquotedblleft orthogonal degree\textquotedblright\ by
Rybicki in \cite{Ry94}. See also \cite{Da85} and \cite{IMV89} for the case of
gradients. The general case of the action of abelian groups was treated in
\cite{IzVi99}. The complete study of the orthogonal degree theory is given in
\cite{IzVi03}, Chapters 2 and 4. From the theoretical point of view, the
theory has to be extended to the action of non-abelian groups and to abstract
infinite dimensional spaces .
\end{remark}

\section{Satellite}

The restricted $n$-body problem is the study of the movement of a satellite
attracted by $n$ primary bodies which are rotating, at a constant angular
speed, around an axis. Since the mass of the satellite is small, one assumes
that the satellite does not perturb the trajectories of the primaries, which
follow the trajectories of relative equilibrium and, as such, are in a plane.

Let $q(t)\in\mathbb{R}^{3}$ be the position of the satellite without mass, and
let $(a_{j},0)$ be the position of a primary body with mass $m_{j}$. Let $J$
be the standard symplectic matrix in $\mathbb{R}^{2}$. In rotating coordinates
$q(t)=(e^{\omega Jt}u(t),z(t))$, $u\in\mathbb{R}^{2}$, Newton's equations
describing the movement of the satellite, with angular speed $\omega=1$, are%
\begin{align}
\ddot{u}+2J\dot{u}-u  &  =-\sum_{j=1}^{n}m_{j}\frac{u-a_{j}}{\left\Vert
(u,z)-(a_{j},0)\right\Vert ^{3}}\text{,}\label{sat}\\
\ddot{z}  &  =-\sum_{j=1}^{n}m_{j}\frac{z}{\left\Vert (u,z)-(a_{j}%
,0)\right\Vert ^{3}}\text{.}\nonumber
\end{align}

One may ask for existence of bifurcation of periodic solutions starting from
the equilibria of the satellite. These solutions will form a continuum in the
plane of the primaries and there are other global branches outside of that
plane. The proof is based on the use of the orthogonal degree.

\subsection{The orthogonal bifurcation map}

Let $H_{2\pi}^{2}(\mathbb{R}^{n})$ be the Sobolev space of $2\pi$-periodic
functions. Define the collision points set as $\Psi=\{a_{1},...,a_{n}\}$, and
the collision-free paths as
\[
H_{2\pi}^{2}(\mathbb{R}^{3}\backslash\Psi)=\{x\in H_{2\pi}^{2}(\mathbb{R}%
^{3}):x(t)\neq a_{j}\}.
\]
Changing variables from $t$ to $t/\nu$, the $2\pi/\nu$-periodic solutions are
zeros of the map%
\begin{align*}
f  &  :H_{2\pi}^{2}(\mathbb{R}^{3}\backslash\Psi)\times\mathbb{R}%
^{+}\rightarrow L_{2\pi}^{2}\\
f(x,\nu)  &  =-\nu^{2}\ddot{x}-2\nu\emph{diag}(J,0)\dot{x}+\nabla V(x)\text{.}%
\end{align*}
where $V$ is the potential
\[
V(u,z)=\left\vert u\right\vert ^{2}/2-\sum_{j=1}^{n}m_{j}\frac{1}{\left\Vert
(u,z)-(a_{j},0)\right\Vert }.
\]

In view of the definitions, the collision-free $2\pi$-periodic solutions are
zeros of the bifurcation operator $f(x,\nu)$. Furthermore, the operator $f$ is
well defined and continuous.

Define the actions of the group $\mathbb{Z}_{2}\times S^{1}$ on $H_{2\pi}%
^{2}(\mathbb{R}^{3}\backslash\Psi)$ as
\[
\rho(\kappa)x=\emph{diag}(1,1,-1)x\text{ and }\rho(\varphi)x=x(t+\varphi).
\]
Since the equation of the satellite is invariant by this reflection, and since
the equation is autonomous, then $f$ is $\mathbb{Z}_{2}\times S^{1}%
$-equivariant. The generator of the group $S^{1}$ in the space $H_{2\pi}^{2}$
is
\[
Ax=\frac{d}{d\varphi}(\rho(\varphi)x)_{\varphi=0}=\dot{x}.
\]
Moreover, the map $f$ is $\mathbb{Z}_{2}\times S^{1}$-orthogonal because it
satisf{ies} the orthogonal condition
\[
\left\langle f(x),\dot{x}\right\rangle _{L_{2\pi}^{2}}=\int_{0}^{2\pi}\left(
-\nu^{2}\left\vert \dot{x}\right\vert ^{2}/2+V(x)\right)  ^{\prime
}dt=0\text{.}%
\]

\begin{remark}
For periodic and non-periodic solutions of the equations, the conservation of
energy is written as
\[
E=-\nu^{2}\left\vert \dot{x}\right\vert ^{2}/2+V(x)=cte.
\]
Thus, one may think that the orthogonal condition is equivalent to
conservation of energy.
\end{remark}

The Fourier transform of the bifurcation map is
\[
f(x)=\sum_{l\in\mathbb{Z}}\left(  l^{2}\nu^{2}x_{l}-2il\nu\emph{diag}%
(J,0)x_{l}+g_{l}\right)  e^{ilt}\text{,}%
\]
where $x_{l}$ and $g_{l}$ are the Fourier modes of $x$ and $\nabla V(x)$ respectively.

Since the matrix
\[
l^{2}\nu^{2}I-2il\nu\emph{diag}(J,0)
\]
is invertible for all $l$'s, except a finite number. One may perform a global
Lyapunov-Schmidt reduction using the global implicit function theorem for
non-collision paths. In this way, one gets the reduced map $f_{1}(x_{1}%
,x_{2}(x_{1},\nu),\nu)$, where $x_{1}$ corresponds to a finite number of modes
and $x_{2}$ to the complement. Moreover, the reduced map is a $\Gamma
$-orthogonal map, see \cite{IzVi03} or \cite{Ga10} for details. Furthermore,
for bifurcation without resonances one may reduce the map to the principal
Fourier mode $l=1$.

For isolated orbits $\Gamma x_{0}$, the degree is calculated in terms of the
linearization at $x_{0}$. Close to an equilibrium $x_{0}$ one has that $\nabla
V(x_{0}+h)=D^{2}V(x_{0})h+o(h)$, then the linearization of the reduced map is%
\begin{align*}
f_{1}^{\prime}(x_{0},\nu)x_{1}  &  =\sum_{\text{finite }l^{\prime}\text{s}%
}M(l\nu)x_{l}e^{ilt}\text{ with }\\
M(\nu)  &  =\nu^{2}I-2i\nu\emph{diag}(J,0)+D^{2}V(x_{0})\text{.}%
\end{align*}
So the linearization of the reduced map is a diagonal matrix with blocks
$M(l\nu)$ for a finite number of $l$'s. For bifurcation without resonances, it
has only the block $M(\nu)$, for the $1$-th Fourier mode.

\subsection{Symmetries}

The action of the element $(\kappa,\varphi)\in\mathbb{Z}_{2}\times S^{1}$
satisfy%
\[
\rho(\kappa,\varphi)x=\rho(\kappa)x(t+\varphi)=\sum_{l}\rho(\kappa
)e^{il\varphi}x_{l}e^{ilt}\text{,}%
\]
thus the action of the group is inherit{ed} on the Fourier modes as
\[
\rho(\kappa,\varphi)x_{l}=\rho(\kappa)e^{il\varphi}x_{l}\text{.}%
\]

Since all the equilibria are planar, the isotropy subgroup of any equilibrium
is $\mathbb{Z}_{2}\times S^{1}$, this means that all equilibria are fixed by
the action of $\mathbb{Z}_{2}\times S^{1}$. When one apply orthogonal degree
to the bifurcation problem, one need to know the irreducible representations
of the action of $\Gamma_{x_{0}}=\mathbb{Z}_{2}\times S^{1}$.

\subsubsection{Planar symmetries}

{In order to simplify the exposition, only the symmetries of the group
$\mathbb{Z}_{2}\times S^{1}$ for the $1$-th mode will be studied. This
correspond to the case without resonances. The space $\mathbb{C}^{3}$
corresponding to the $1$-th mode has two spaces of similar irreducible
representations: $V_{0}=\mathbb{C}^{2}\times\{0\}$ and $V_{1}=\{0\}\times
\mathbb{C}$. This is, the group $\mathbb{Z}_{2}$ acts as $\rho(\kappa)=I$ on
$V_{0}$ , and as $\rho(\kappa)=-1$ on $V_{1}$ .} Consequently, the action of
the group $\mathbb{Z}_{2}\times S^{1}$ in $V_{0}$ for the $1$-th mode is
\[
\rho(\kappa,\varphi)x=e^{i\varphi}x.
\]
Since $(\kappa,0)$ is the only element that fix the points of $V_{0}$, the
isotropy subgroup of the points in $V_{0}$ is generated by $(\kappa,0)$,%
\[
\mathbb{Z}_{2}=\left\langle (\kappa,0)\right\rangle .
\]

Solutions $x=(u,z)$ to the equation (\ref{sat}) with isotropy group
$\mathbb{Z}_{2}$ satisfy%
\[
x(t)=\rho(\kappa)x(t)=\emph{diag}(1,1,-1)x(t)\text{.}%
\]
Therefore, solutions to the equation (\ref{sat}) with symmetry $\mathbb{Z}%
_{2}$ are just planar solutions, i.e. $z(t)=0$.

\subsubsection{Spatial symmetries}

In $V_{1}$ the action of the group $\mathbb{Z}_{2}\times S^{1}$ is
\[
(\kappa,\varphi)x=-e^{i\varphi}x.
\]
Since $(\kappa,\pi)$ is the only element that fix the points of $V_{1}$, thus
the isotropy subgroup for $V_{1}$ is generated by $(\kappa,\pi)$,
\[
\mathbb{\tilde{Z}}_{2}=\left\langle (\kappa,\pi)\right\rangle .
\]

Solutions $x=(u,z)$ to the equation (\ref{sat}) with isotropy group
$\mathbb{\tilde{Z}}_{2}$ satisfy
\[
x(t)=\rho(\kappa,\pi)x(t)=\emph{diag}(1,1,-1)x(t+\pi)\text{,}%
\]
this is%
\[
u(t)=u(t+\pi)\text{ and }z(t)=-z(t+\pi)\text{.}%
\]

Solutions to the equation (\ref{sat}) with these symmetries follows twice the
planar $\pi$-periodic curve $u$, one time with the spatial coordinate $z$ and
a second time with $-z$. Consequently, there is at least one $t_{0}$ where
$z(t_{0})=z(t_{0}+\pi)=0$. For instance, if only one of these zeros exists,
then the solution looks like a spatial eight near the equilibrium. For this
reason, these solutions will be called eight-solutions.

\subsection{Bifurcation theorem}

For bifurcation without resonances, one may reduce the bifurcation study to
the $1$-th Fourier mode. In this case, the $\mathbb{Z}_{2}\times S^{1}%
$-orthogonal degree of the reduced map complemented by the right function is%
\[
\eta_{\mathbb{Z}_{2}}(\nu_{0})[F_{\mathbb{Z}_{2}}]+\eta_{\mathbb{\tilde{Z}%
}_{2}}(\nu_{0})[F_{\mathbb{\tilde{Z}}_{2}}]\text{,}%
\]
where $[F_{\mathbb{Z}_{2}}]$ and $[F_{\mathbb{\tilde{Z}}_{2}}]$ are generators
of one $\mathbb{Z}$ in the homotopy group $\Pi_{\perp}$. The numbers
$\eta_{\ast}(\nu_{0})$ correspond to the change of Morse index of the block
$M(\nu)$ in the space $V_{0}$, for $\mathbb{Z}_{2}$, and in the space $V_{1}$,
for $\mathbb{\tilde{Z}}_{2}$.

From the existence property of the degree, one has a zeros of the bifurcation
map when $\eta(\nu_{0})\neq0$, this is, there is periodic solutions near
$(x_{0},\nu_{0})$ with isotropy group $\mathbb{Z}_{2}$, if $\eta
_{\mathbb{Z}_{2}}(\nu_{0})\neq0$, and with isotropy group $\mathbb{\tilde{Z}%
}_{2}$, if $\eta_{\mathbb{\tilde{Z}}_{2}}(\nu_{0})\neq0$. For resonances one
may have more generators of $\Pi_{\perp}$ corresponding to bifurcation of
harmonic periods of the principal one.

What remains is to analyze the Morse index in the subspaces $V_{0}$ and
$V_{1}$. This is done in \cite{GaIz10}, where one arrives at the following conclusion.

\begin{theorem}
Let $T$\ and $D$\ be the trace and determinant of the Hessian of the potential
in the plane, $V$, at the equilibrium $x_{0}$. If $D<0$, there is one global
bifurcation of planar periodic solutions from $x_{0}$. If $0<D<(2-T/2)^{2}$,
there are two global bifurcations of planar periodic solutions.
\end{theorem}

\begin{theorem}
Every equilibrium $x_{0}$ has a global bifurcation of periodic eight solutions%
\[
u(t)=u(t+\pi)\text{ and }z(t)=-z(t+\pi)\text{.}%
\]
Moreover, the local branch is truly spatial, $z(t)\neq0$, provided that some
nonresonant condition between the periods of the spatial and the planar
solutions is satisfy.
\end{theorem}

By global branch, one means that there is a continuum of solutions starting at
the equilibrium, where the continuum goes to infinity in the norm of the
solution or in the period , or goes to collision, or otherwise goes to other
relative equilibria in such a way that the sum of the jumps in the orthogonal
degrees is zero.

\subsubsection{A Morse potential}

One may easily prove that all equilibria for the satellite are planar.
Moreover, provided that the potential in the plane $V$ is a Morse function,
there are at least one global minimum and $n$ saddle points, see
\cite{GaIz10}. For example, in the classical restricted three body problem,
case $n=2$, there are two minimums where the satellite form an equilateral
triangle with the primaries, and three saddle points where the satellite is
collinear with the two primaries.

\begin{theorem}
Provided that the potential in the plane $V$ is a Morse function, each one of
the $n$ saddle points has one global bifurcation of planar periodic solutions,
and one global bifurcation of periodic eight solutions.
\end{theorem}

\begin{theorem}
The minimum point satisfy one of the following options:\ (a) it has two global
bifurcations of planar periodic solutions and one bifurcation of periodic
eight solutions, or (b) it has only one bifurcation of spatial periodic eight solutions.
\end{theorem}

\subsubsection{The Maxwell's Saturn ring}

One may apply these results when the primaries form the Maxwell's Saturn ring,
see { Proposition }\ref{ring}. This is a classical model for Saturn and one
ring around it. In this case one has the following theorem.

\begin{theorem}
The potential has two $\mathbb{Z}_{n}$-orbits of saddle points (r1) and (r2),
when $n\geq2$, and one more $\mathbb{Z}_{n}$-orbit of saddle points when
$n\geq3$ and $\mu$ is near from zero. Furthermore, each saddle point has one
global bifurcation of planar periodic solutions and one global bifurcation of
periodic eight-solutions.
\end{theorem}

\begin{theorem}
The potential has one $\mathbb{Z}_{n}$-orbit of minimum points (r3) for
$n\geq2$. Moreover, provided $\mu$ is big enough, each minimum point has two
global bifurcations of planar periodic solutions, and one global bifurcation
of periodic eight solutions. On the other hand, if $\ $ $\mu$ is small and
$n\geq3$, there is another $\mathbb{Z}_{n}$-orbit of minimum points with only
one bifurcation of spatial periodic eight-solutions.
\end{theorem}

\begin{figure}[h]
\centering
\par
\subfloat[$\mu\in(0,\infty)$] {\
\begin{pspicture}(-2,-2)(2,2)\SpecialCoor
\psarc{->}(0,0){2}{25}{45}
\pscircle[linestyle=dashed](0,0){1}
\psdots(0,0)(1;-120)(1,0)(1;120)
\psdots[dotstyle=otimes](1.5;-120)(1.5,0)(1.5;120)
\psdots[dotstyle=o](1.5;60)(1.5;180)(1.5;-60)
\psdots[dotstyle=otimes](.5;-120)(.5,0)(.5;120)
\rput[bl](1.5,0){\small $r_{1}$}
\rput[bl](.5,0){\small $r_{2}$}
\rput[bl](1.5;60){\small $r_{3}$}
\NormalCoor\end{pspicture}
} \qquad\subfloat[$\mu=0$] {\
\begin{pspicture}(-2,-2)(2,2)\SpecialCoor
\psarc{->}(0,0){2}{25}{45}
\pscircle[linestyle=dashed](0,0){1}
\psdots(1;-120)(1,0)(1;120)
\psdots[dotstyle=otimes](1.5;-120)(1.5,0)(1.5;120)
\psdots[dotstyle=o](0,0)(1.5;60)(1.5;180)(1.5;-60)
\psdots[dotstyle=otimes](.7;60)(.7;180)(.7;-60)
\rput[bl](1.5,0){\small $r_{1}$}
\rput[l](.5;60){\small $r_{2}$}
\rput[bl](1.5;60){\small $r_{3}$}
\NormalCoor\end{pspicture}
}\caption{Example for $n=3$.}%
\end{figure}
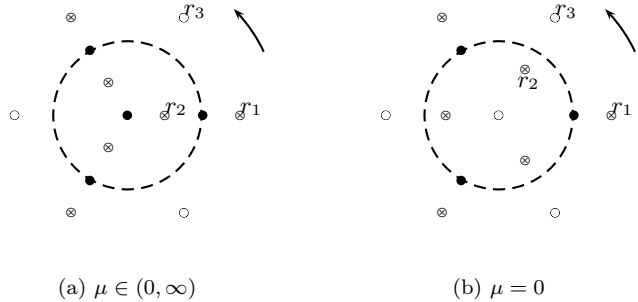

\begin{remark}
The equilibria of the $\mathbb{Z}_{n}$-orbit of minimum points (r3) are
linearly stable if $\mu$ is big enough. This is proven in the paper
\cite{BaEl04}. The existence of the two extra orbits of equilibria for $\mu$
small was pointed out in the paper \cite{ArEl04}. The stability and this fact
is proven also in the paper \cite{GaIz10}. The orthogonal degree has been used
to prove bifurcation in the restricted three body problem also in the paper
\cite{MR04}.
\end{remark}

\begin{remark}
The degree arguments, coupled with group representation ideas, give global
information, i.e., an indication of where the bifurcation branches could go.
Also, since the results are valid for problems which are deformation of the
original problem, the method does not require high order computations and they
may be applied in some degenerate cases (for instance it is not necessary that
the bifurcation parameter crosses a critical value with non-zero speed; it {
is} enough that it crosses {it } eventually). An immediate drawback of this
approach is that topological methods do not provide a detailed information on
the local behavior of the bifurcating branch, such as stability or the
existence of other type of solutions, like KAM tori. Other methods, such as
normal forms or special coordinates, should be used for these purposes but
they only provide local information near the critical point. In a similar way,
the degree arguments give only partial results on resonances and other tools
should be used.
\end{remark}

\section{The $n$-body problem}

Let $q_{j}(t)\in\mathbb{R}^{3}$ be the position of the $j$-th body with mass
$m_{j}$, for $j\in\{0,1,...,n\}$. Let $J$ be the standard symplectic matrix in
$\mathbb{R}^{2}$. Newton's equations of the $n$ bodies, in rotating
coordinates $q_{j}(t)=(e^{\sqrt{\omega}tJ}u_{j}(t),z_{j}(t))$, are%
\begin{align}
m_{j}\ddot{u}_{j}+2m_{j}\sqrt{\omega}J\dot{u}_{j}  &  =\omega m_{j}u_{j}%
-\sum_{i=0(i\neq j)}^{n}m_{i}m_{j}\frac{u_{j}-u_{i}}{\left\Vert (u_{j}%
,z_{j})-(u_{i},z_{i})\right\Vert ^{3}}\label{nbod}\\
m_{j}\ddot{z}_{j}  &  =-\sum_{i=0(i\neq j)}^{n}m_{i}m_{j}\frac{z_{j}-z_{i}%
}{\left\Vert (u_{j},z_{j})-(u_{i},z_{i})\right\Vert ^{3}}\text{.}\nonumber
\end{align}

Relative equilibria of the $n$-body problem correspond to equilibria in these
rotating coordinates. Since all relative equilibria are planar, the positions
$(a_{j},0)$ correspond to a relative equilibrium if they satisfy the relations%
\begin{equation}
\omega a_{j}=\sum_{i=0(i\neq j)}^{n}m_{i}\frac{a_{j}-a_{i}}{\left\Vert
a_{j}-a_{i}\right\Vert ^{3}}\text{.} \label{re}%
\end{equation}

\begin{remark}
Actually, identifying the plane and the complex plane, solutions of (\ref{re})
may give also homographic solutions of the form $q_{j}=qa_{j}$, where the
function $q(t)\in\mathbb{C}$ satisfy the Kepler equation. In these general
solutions, the bodies may move in ellipses, parabolas or hyperbolas, instead
of circular orbits. One may have also solutions with total collapse or growing
like $q(t)=(9\omega/2)^{1/3}t^{2/3}$.
\end{remark}

\begin{proposition}
\label{ring}Set the position of the bodies as: $a_{0}=0$ with mas $\mu$, and
$a_{j}=e^{ij\zeta}$ with mass $1$ for $j\in\{1,...,n\}$, where $\zeta=2\pi/n$.
The $a_{j}$'s correspond to a relative equilibrium when $\omega=\mu+s_{1}$,
where%
\[
s_{1}=\sum_{j=1}^{n-1}\frac{1-e^{ij\zeta}}{\left\Vert 1-e^{ij\zeta}\right\Vert
^{3}}=\frac{1}{4}\sum_{j=1}^{n-1}\frac{1}{\sin(j\zeta/2)}.
\]

\end{proposition}

\begin{proof}
For $j=0$ the equality is $\omega a_{0}-\mu\sum_{j=0}^{n-1}e^{ij\zeta}=0$. For
$j\neq0$ the equality is%
\[
\sum_{i=1~(i\neq j)}^{n}\frac{a_{j}-a_{i}}{\left\Vert a_{j}-a_{i}\right\Vert
^{3}}+\mu a_{j}=(\mu+s_{1})a_{j}=\omega a_{j},
\]

\end{proof}

Therefore, the $a_{j}$'s form a relative equilibrium for the frequency
$\omega=\mu+s_{1}$. This relative equilibrium was studied by Maxwell as a
simplified model of Saturn and its rings.

In the paper \cite{GaIz11}, Proposition 23, one finds that for each
$k\in\{1,...,n-1\}$, there is one mass $\mu_{k}$ with one global bifurcation
of relative equilibria. Let $h$ be the maximum common divisor of $k$ and $n$,
the bifurcation branch from $\mu_{k}$ has solutions where $n$ bodies are
arranged as $n/h$ regular polygons of $h$ sides. See the example for $n=6$.

\begin{figure}[h]
\centering\subfloat[
Symmetries for $k=1$.] {
\begin{pspicture}(-2,-2)(2,2)\SpecialCoor
\psline{->}(-2,0)(2,0)
\pspolygon[showpoints=true,linestyle=dashed](1,0)(1.5;70)(1;140)(-2,0)(1;-140)(1.5;-70)
\psdots(.5,0)
\NormalCoor\end{pspicture}
} \subfloat[
Symmetries for  $k=2$.] {
\begin{pspicture}(-2,-2)(2,2)\SpecialCoor
\psline{->}(-2,0)(2,0)
\pspolygon[showpoints=true,linestyle=dashed](2,0)(1.3;30)(1.3;150)(-2,0)(1.3;-150)(1.3;-30)
\psdots(0,0)
\NormalCoor\end{pspicture}
} \subfloat[
Symmetries for $k=3$.] {
\begin{pspicture}(-2,-2)(2,2)\SpecialCoor
\psline{->}(-2,0)(2,0)
\pspolygon[showpoints=true,linestyle=dashed](2;0)(.5;60)(2;120)(.5;180)(2;240)(.5;300)
\psdots(0,0)
\NormalCoor\end{pspicture}
}\caption{ $n=6$.}%
\end{figure}

\begin{remark}
The $n$-body problem has been the object of many papers, with different
techniques and different purposes. For the stability of the polygonal
equilibrium, or the bifurcation of relative equilibria from it, one shall
mention: \cite{VaKo07}, \cite{Ro00}, \cite{MeSc88}, \cite{Sc03}, among others.
\end{remark}

\subsection{The orthogonal bifurcation map}

Changing variables from $t$ to $t/\nu$, the $2\pi/\nu$-periodic solutions of
equation (\ref{nbod}) are zeros of the bifurcation map $f$ defined in the
spaces%
\[
f:H_{2\pi}^{2}(\mathbb{R}^{3(n+1)}\backslash\Psi)\times\mathbb{R}%
^{+}\rightarrow L_{2\pi}^{2}\text{,}%
\]
where $\Psi=\{x\in\mathbb{R}^{3(n+1)}:x_{i}=x_{j}\}$ is the collision set,
corresponding to two or more of the bodies colliding, and $H_{2\pi}%
^{2}(\mathbb{R}^{3(n+1)}\backslash\Psi)$ is the open subset, consisting of the
collision-free periodic (and continuous) functions of the Sobolev space
$H_{2\pi}^{2}(\mathbb{R}^{3(n+1)})$.

Define the action of $(\kappa,\theta)\in\mathbb{Z}_{2}\times SO(2)$ in
$\mathbb{R}^{3(n+1)}$ as
\begin{align*}
\rho(\kappa)(u_{j},z_{j})  &  =(u_{j},-z_{j})\text{, }\\
\rho(\theta)(u_{j},z_{j})  &  =(e^{-J\mathcal{\theta}}u_{j},z_{j})\text{,}%
\end{align*}
where the group $\mathbb{Z}_{2}$ reflects the $z$-axis, and where $SO(2)$
rotates the $(x,y)$-plane. Since { }Newton's equations are invariant by
isometries, the group $\mathbb{Z}_{2}\times SO(2)$ represent{s} the inherited
isometries in rotating coordinates and the map $f$ is $\mathbb{Z}_{2}\times
SO(2)$-equivariant .

Let $S_{n}$ be the group of permutations of the numbers $\{1,...,n\}$. Define
the action of an element $\gamma\in$ $S_{n}$ in $x\in\mathbb{R}^{3(n+1)}$ as
$\rho(\gamma)x_{0}=x_{0}$ for $j=0$, and for $j\in\{1,...,n\}$ as%
\[
\rho(\gamma)x_{j}=x_{\gamma(j)}\text{.}%
\]
Since the action of $S_{n}$ permutes the $n$ bodies with equal mass, then the
map $f$ is $S_{n}$-equivariant.

The map $f$ is $S^{1}$-equivariant with the action $\rho(\varphi
)x(t)=x(t+\varphi)$, because the equations are autonomous. As the orthogonal
degree is defined only for abelian groups, the map $f$ will be considered only
as $\Gamma\times S^{1}$-equivariant, where $\Gamma$ is the abelian group%
\[
\Gamma=\mathbb{Z}_{2}\times\mathbb{Z}_{n}\times SO(2),
\]
and $\mathbb{Z}_{n}$ is the subgroup of $S_{n}$ generated by $\zeta(j)=j+1$.

The element $\kappa\in\mathbb{Z}_{2}$ always leaves an equilibrium fixed
because all equilibria are planar, see \cite{Ga10} for a proof. Let
$\mathbb{\tilde{Z}}_{n}$ be the subgroup of $\Gamma$ generated by
\[
(\zeta,\zeta)\in\mathbb{Z}_{n}\times SO(2)\text{,}%
\]
where $\zeta=2\pi/n\in SO(2)$. The actions of $(\zeta,\zeta)$ send the point
$x_{0}$ to $e^{-J\zeta}x_{0}$, and it sends $x_{j}$ to $e^{-J\zeta}x_{j+1}$for
the other $j$'s . One may easily verify that the $a_{j}$'s are fixed by the
action of $(\zeta,\zeta)$, thus the isotropy group of $a$ is the group
${\Gamma}_{a}\times S^{1}$ with%
\[
\Gamma_{a}=\mathbb{Z}_{2}\times\mathbb{\tilde{Z}}_{n}.
\]

In each component, the infinitesimal generator of the action of $S^{1}$ is
given by $A_{0}x_{j}=\dot{x}_{j}$, and the infinitesimal generators of $SO(2)$
is given by
\[
A_{1}x_{j}=\frac{\partial}{\partial\theta}|_{\theta=0}(e^{-J\mathcal{\theta}%
}u_{j},z_{j})=diag(-J,0)x_{j}\text{.}%
\]
Thus, the equalities $\left\langle f(x),\dot{x}\right\rangle _{L_{2\pi}^{2}%
}=0$ and $\left\langle f(x),A_{1}x\right\rangle _{L_{2\pi}^{2}}=0$ follow as
the proof of conservation of energy and angular momentum for Newton's
equations, see \cite{GaIz13} for a proof. Thus the map $f$ is a $\Gamma\times
S^{1}$-orthogonal map.

\begin{remark}
The orbit of the polygonal equilibrium $a$ consists of all the rotations in
the $(x,y)$-plane. Since $f=0$ on the orbit $\Gamma a$, deriving the map $f$
along a parametrization of this orbit one gets that the generator $A_{1}a$ is
tangent to the orbit, and must be in the kernel of $f^{\prime}(a)$. This is a
well known fact where symmetries imply degeneracies.
\end{remark}

\subsection{Symmetries}

\subsubsection{Planar symmetries}

In the paper \cite{GaIz11}, it is proven that there are $n$ subspaces $W_{k}$
for the similar irreducible representations of $\mathbb{\tilde{Z}}_{n}$, where
the action of $\kappa\in\mathbb{Z}_{2}$ is $\rho(\kappa)=I$, and the action of
$(\zeta,\zeta)\in\mathbb{\tilde{Z}}_{n}\ $ is given by
\[
\rho(\zeta,\zeta,\varphi)=e^{ik\zeta}\text{.}%
\]

Moreover, since the action of $S^{1}$ on the fundamental Fourier mode is given
by
\[
\rho(\varphi)=e^{i\varphi}\text{,}%
\]
the isotropy subgroup of $\Gamma_{\bar{a}}\times S^{1}$ in the space $W_{k}$
is generated by $\kappa\in\mathbb{Z}_{2}$ and $\left(  \zeta,\zeta
,-k\zeta\right)  \in\mathbb{\tilde{Z}}_{n}\times S^{1}$. This is, the points
of $W_{k}$ are fixed by the group%

\[
\mathbb{\tilde{Z}}_{n}(k)\times\mathbb{Z}_{2}=\left\langle \left(  \zeta
,\zeta,-k\zeta\right)  \right\rangle \times\left\langle \kappa\right\rangle .
\]

As for the satellite, solutions with isotropy group $\mathbb{Z}_{2}$ must
satisfy $z_{j}(t)=0$, and solutions with isotropy group $\mathbb{\tilde{Z}%
}_{n}(k)$ satisfy the symmetries
\[
u_{j}(t)=\rho(\zeta,\zeta,-k\zeta)u_{j}(t)=e^{-i\zeta}u_{\zeta(j)}%
(t-k\zeta)\text{.}%
\]

In this case, for the central body one has the symmetry
\[
u_{0}(t)=e^{\ ij\zeta}u_{0}(t+jk\zeta).
\]
Using the notation $u_{j}=u_{j+kn}$ for $j\in\{1,...,n\}$, one has that
$\zeta(j)=j+1$, then the $n$ bodies with equal mass satisfy%
\[
u_{j+1}(t)=e^{\ ij\zeta}u_{1}(t+jk\zeta).
\]
Thus, each one of the $n$ bodies with equal mass follows the same planar
curve, but with different phase and with some rotation in the $(x,y)$-plane.

\begin{figure}[h]
\centering
\subfloat[Symmetries of $\mathbb{\tilde{Z}}_{n}(1)$.] {
\begin{pspicture}(-2.5,-2.5)(2.5,2.5)\SpecialCoor
\psdots[dotstyle=o](2;0)(2;72)(2;146)(2;219)(2;292)(0,0)
\psline[linestyle=dashed](.5;0)(.5;72)(.5;146)(.5;219)(.5;292)
\psline{*->}(.5;0)(.5;292)
\psellipticarc[linestyle=dashed](2;0)(.5,.3){0}{292}
\psellipticarc{*->}(2;0)(.5,.3){292}{0}
\rput{72}{
\psellipticarc[linestyle=dashed](2;0)(.5,.3){72}{0}
\psellipticarc{*->}(2;0)(.5,.3){0}{72}}
\rput{146}{
\psellipticarc[linestyle=dashed](2;0)(.5,.3){146}{72}
\psellipticarc{*->}(2;0)(.5,.3){72}{146}}
\rput{219}{
\psellipticarc[linestyle=dashed](2;0)(.5,.3){219}{146}
\psellipticarc{*->}(2;0)(.5,.3){146}{219}}
\rput{292}{
\psellipticarc[linestyle=dashed](2;0)(.5,.3){292}{219}
\psellipticarc{*->}(2;0)(.5,.3){219}{292}}
\NormalCoor\end{pspicture}
} \qquad{} \subfloat[Symmetries of $\mathbb{\tilde{Z}}_{n}(2)$.] {
\begin{pspicture}(-2.5,-2.5)(2.5,2.5)\SpecialCoor
\psdots[dotstyle=o](2;0)(2;72)(2;146)(2;219)(2;292)(0,0)
\psline{*->}(.5;0)(.5;146)
\psline(.5;146)(.5;292)(.5;72)(.5;219)(.5;0)
\psellipticarc[linestyle=dashed](2;0)(.5,.3){0}{292}
\psellipticarc{*->}(2;0)(.5,.3){292}{0}
\rput{72}{
\psellipticarc[linestyle=dashed](2;0)(.5,.3){146}{72}
\psellipticarc{*->}(2;0)(.5,.3){72}{146}}
\rput{146}{
\psellipticarc[linestyle=dashed](2;0)(.5,.3){292}{219}
\psellipticarc{*->}(2;0)(.5,.3){219}{292}}
\rput{219}{
\psellipticarc[linestyle=dashed](2;0)(.5,.3){72}{0}
\psellipticarc{*->}(2;0)(.5,.3){0}{72}}
\rput{292}{
\psellipticarc[linestyle=dashed](2;0)(.5,.3){219}{146}
\psellipticarc{*->}(2;0)(.5,.3){146}{219}}
\NormalCoor\end{pspicture}
}\caption{For $n=5$.}%
\end{figure}

\begin{remark}
In fixed coordinates, the solutions are $q_{j}(t)=e^{i\sqrt{\mathcal{\omega}%
}t}u_{j}(\nu t)$. Thus in fixed coordinates the solutions are in general
{q}uasiperiodic solutions. In particular, when the central body has mass zero,
we are considering the $n$-body problem with equal masses. In this case, one
has for $j\in\{1,...,n\}$ that%
\[
q_{j+1}(t)=e^{\ ij\zeta\Omega}q_{1}(t+jk\zeta)
\]
with $\Omega=1-k\sqrt{\omega}/\nu$. If $\Omega\in n\mathbb{Z}$, then solutions
with isotropy group $\mathbb{\tilde{Z}}_{n}(k)$ satisfy
\[
q_{j+1}(t)=q_{1}(t+jk\zeta).
\]
These solutions where all the bodies follow the same path are known as
choreographies, see \cite{CM00}.
\end{remark}

\subsubsection{Spatial symmetries}

In the paper \cite{GaIz13} it is proven that there are $n$ subspaces $W_{k}$
for the similar irreducible representations of $\mathbb{\tilde{Z}}_{n}$, where
the action of $\kappa\in\mathbb{Z}_{2}$ is given by $\rho(\kappa)=-I$, and the
action of the element $(\zeta,\zeta)\in\mathbb{\tilde{Z}}_{n}$ is%
\[
\rho(\zeta,\zeta,\varphi)=e^{ik\zeta}.
\]

Since the action of $S^{1}$ on the fundamental mode is $\rho(\varphi
)=e^{i\varphi}$, then the elements $(\zeta,\zeta,-k\zeta)\in\mathbb{\tilde{Z}%
}_{n}\times S^{1}$ and $(\kappa,\pi)\in\mathbb{Z}_{2}\times S^{1} $ act
trivially on $W_{k}$. Thus, the isotropy group of $W_{k}$ is generated by
$(\zeta,\zeta,-k\zeta)\ $and $(\kappa,\pi)$,%
\[
\mathbb{\tilde{Z}}_{n}(k)\times\mathbb{\tilde{Z}}_{2}=\left\langle \left(
\zeta,\zeta,-k\zeta\right)  \right\rangle \times\left\langle (\kappa
,\pi)\right\rangle .
\]

As we saw for the satellite, solutions with isotropy group $\mathbb{\tilde{Z}%
}_{2}$ satisfy%
\[
u_{j}(t)=u_{j}(t+\pi)\text{ and }z_{j}(t)=-z_{j}(t+\pi)\text{,}%
\]
thus the projection of this solution on the $(x,y)$-plane follows twice the
$\pi$-periodic curve $u(t)$, one time with the spatial coordinate $z(t)$ and a
second time with $-z(t)$. Thus solution looks like a spatial eight near the equilibrium.

Since the group $\mathbb{\tilde{Z}}_{n}(k)$ is generated by $(\zeta
,\zeta,-k\zeta)$, the solutions satisfy also the symmetries
\begin{align*}
u_{j}(t)  &  =e^{-i\zeta}u_{\zeta(j)}(t-k\zeta)\text{, }\\
z_{j}(t)  &  =z_{\zeta(j)}(t-k\zeta)\text{.}%
\end{align*}

\begin{remark}
To see one example, suppose that $n=2m$ and choose $k=m$. In this case the
central body remains at the center. Moreover, the $n$ bodies with equal masses
satisfy
\[
u_{j+1}(t)=e^{\ ij\zeta}u_{1}(t+j\pi)=e^{\ ij\zeta}u_{1}(t)
\]
and%
\[
z_{j+1}(t)=z_{1}(t+j\pi)=(-1)^{j}z_{1}(t).
\]
Thus, there are two $m$-polygons which oscillate vertically, one with
$z_{1}(t)$ and the other with $-z_{1}(t)$. Furthermore, the projection of the
two $m$-polygons in the plane is always a $2m$-polygon. These solutions are
known as Hip-Hop orbits.
\end{remark}

See \cite{GaIz13} for a general description of the symmetries

\subsection{Bifurcation theorem}

The linearization of the system at the polygonal equilibrium is a
${3(n+1)}\times{3(n+1)}$ matrix, which is non invertible due to the rotational
symmetry. In \cite{GaIz13} one finds a change of variables that organize the
spaces $W_{k}$'s of similar irreducible representation of $\Gamma_{a}\times
S^{1}$, and also simplify the analysis of the spectrum . This is, the arrange
of the subspaces of similar irreducible representations gives a decomposition
of the linearization in $2n$ blocks, $n$ of them for the spatial coordinates,
given in \cite{GaIz13}, and $n$ of them for the planar coordinates, given in
\cite{GaIz11}.

Applying orthogonal degree to the reduced bifurcation map, one find{s} that
the degree has one component for each one of these $2n$ blocks, when there
{are} no resonances. In the case of the satellite there were only two
components. Each component has one number $\eta(\nu)$ which is the change of
Morse index of the corresponding block. By the existence property of the
degree, there is one bifurcation branch starting from $(a,\nu_{0})$ each time
$\eta(\nu_{0})\neq0$, with the symmetries of the corresponding block.

In this way one get the following theorems, see \cite{GaIz13} for details.

\begin{theorem}
For $n\geq3$ and each $k\in\{2,...,n-2\}$, the polygonal equilibrium has a
global bifurcation of planar periodic solutions with symmetries
$\mathbb{\tilde{Z}}_{n}(k)$, if $\mu\in(-s_{1},\mu_{k})$, and two global
bifurcations if $\mu\in(m_{+},\infty)$.

For $n\geq7$ and each $k\in\{1,n-1\}$, the polygonal equilibrium has two
global bifurcation branches of planar periodic solutions with symmetries
$\mathbb{\tilde{Z}}_{n}(k)$ when $\mu>m_{+}$.
\end{theorem}

By global branch, one means that there is a continuum of solutions starting at
the ring configuration, and the continuum goes to infinity in the norm of the
solution or in the period, or goes to collision, or otherwise goes to other
relative equilibria in such a way that the sum of the jumps in the orthogonal
degrees is zero.

\begin{theorem}
The polygonal equilibrium has a global bifurcation of periodic solutions with
symmetries\ $\mathbb{\tilde{Z}}_{n}(k)\times\mathbb{\tilde{Z}}_{2}$ for each
$k\in\{1,...,n\}$. Except for a possible finite number of $\mu$'s and
frequencies, for $\mu$ positive, bounded and different from $\mu_{k}$, due to
resonances, these solutions are truly spatial, this is $z_{j}(t)\neq0 $ for
some $j$-th body.
\end{theorem}

\begin{remark}
Here, only the generic cases were exposed for simplicity. In \cite{GaIz13} all
cases of bifurcation from the polygonal equilibrium were studied for $n\geq2$.
In the paper \cite{GaIz13}, there is also a theorem for the general $n$-body
problem, where it is proven that any relative equilibrium has one bifurcation
of spatial like eight solutions, and that generically there are $n-1$ of these bifurcations.
\end{remark}

The planar bifurcation for $k=n$ consists of solutions with $u_{0}(t)=0$ and
$u_{j}(t)=e^{ij\zeta}u_{n}(t)$. This branch was constructed in an explicit way
in \cite{MS93}, by reducing the problem to a $6$-dimensional dynamical system
and a normal form argument.

The spatial bifurcation for $k=n$ is made of solutions where the ring moves as
a whole and the central body makes the contrary movement in order to stabilize
the forces, that is $z_{j}(t)=z_{1}(t)$ for $j\in\{1,...n\}$ and
$z_{0}=-nz_{j}$. This solution was called an oscillating ring in \cite{MS93}.

The spatial bifurcation for $k=n/2$ {has} the symmetries of the well known
Hip-Hop orbits. This kind of solutions appears first in the paper \cite{DTW83}
without the central body. Later on, in \cite{MS93} for a big central body in
order to explain the pulsation of the Saturn ring, where they are called kink
solutions. Finally, there is a proof in \cite{BCPS06} when there is no central body.

\begin{remark}
The same group of symmetries of the polygonal equilibrium for the $n$-body
problem is present also in the papers: In \cite{GaIz13} for charges instead of
bodies. In \cite{GaIz12} for vortices and traveling waves in almost parallel
filaments{, and} in \cite{GaIz14} for a periodic lattice of coupled nonlinear
Schr\"{o}dinger oscillators . Although there are many similarities with the
$n$ body problem, in particular in the change of variables, these results are
of a quite different nature.
\end{remark}

As we saw before, due to the rotational symmetry, the linearization at any
equilibrium has at least one dimensional kernel. In order to find bifurcation
of relative equilibria, in the paper \cite{GaIz11}, one get rid of this
degeneracy looking for solutions in fixed-point subspaces of some reflection,
where one is able to use ordinary degree or another method.

For bifurcation of periodic solutions, the polygonal equilibrium is fixed only
by the action of
\[
\tilde{\kappa}x_{j}(t)=\emph{diag}(1,-1,1)x_{n-j}(-t)\text{,}%
\]
which is a coupling between the reflection on the plane, a reversal of time,
and a permutation of bodies. This is the only reflection able to get ride of
the degeneracy.

However, when one restricts the problem to the fixed-point subspace of
$\tilde{\kappa}$, one may proves bifurcation of periodic solutions only for
the symmetries $k=n$ and $k=n/2$. For the remaining $k$'s, the linearization
on the fixed-point subspace of $\tilde{\kappa}$ is a complex matrix with
non-negative determinant as a real matrix. One could also use the gradient
structure and apply the results for bifurcation based on Conley index.
Actually, analytical studies with normal forms of high order and additional
hypotheses of non-resonance are proposed in \cite{CF08} for these cases.
However, this approach do not provide the proof of the existence of a global
continuum, something which follows from the application of the orthogonal
degree. This fact implies that one may not use a classical degree argument or
other simple analytical proofs to find the solutions presented here.

\begin{remark}
Variational techniques have been quite successful in treating the existence of
closed solutions. In particular, \cite{FT04}, \cite{F06} and \cite{F07},
classify all the possible groups which give periodic solutions which are
minimizers of the action without collisions. Thus, the issue is different from
ours, since one has the proof of the existence of a solution in the large,
with a specific symmetry. For choreographies, following the seminal paper
\cite{CM00}, with no central mass, {there are} studies with more than 3 bodies
in \cite{Ch01} and \cite{BT04}, for instance. In the case of hip-hop
solutions, these methods were successful in \cite{CV00} and \cite{TV07}. One
of the advantages of the orthogonal degree is that it applies to problems
which are not necessarily variational, but present conserved quantities.
\end{remark}

\end{document}